\documentclass{amsart}
\usepackage{amsfonts,amssymb,amsmath,amsthm}
\usepackage{url}
\usepackage{enumerate}

\renewcommand{\le}{\leqslant}
\renewcommand{\ge}{\geqslant}
\newcommand{\bad}{\mathbf{Bad}}
\newcommand{\elc}{\mathbf{Elc}}

\newcommand{\ord}{\mathrm{ord}}

\newcommand{\ZZ}{\mathbb{Z}}
\newcommand{\QQ}{\mathbb{Q}}

\newcommand{\NN}{\mathbb{N}}
\newcommand{\FF}{\mathbb{F}}

\newcommand{\vi}{\mathbf{i}}

\newtheorem{lemma}{Lemma}
\newtheorem{theorem}{Theorem}
\newtheorem{proposition}{Proposition}
\newtheorem{definition}{Definition}
\newtheorem{problem}{Problem}

\newcommand{\bigk}{\mathop{\mathbf{K}}}

\author{Dzmitry Badziahin}
\address{
School of Mathematics and Statistics\\
University of Sydney\\
NSW 2006, Australia} \email{dzmitry.badziahin@sydney.edu.au}

\begin{document}

\title[On $t$-adic Littlewood conjecture for generalised T.-M. functions]{On $t$-adic Littlewood conjecture for generalised Thue-Morse functions}
\begin{abstract}
We consider a Laurent series defined by infinite products $g_u(t) =
\prod_{n=0}^\infty (1 + ut^{-2^n})$, where $u\in \FF$ is a parameter and
$\FF$ is a field. We show that for all $u\in\QQ\setminus\{-1,0,1\}$ the
series $g_u(t)$ does not satisfy the $t$-adic Littlewood conjecture. On the
other hand, if $\FF$ is finite then $g_u(t)\in \FF((t^{-1}))$ is either a
rational function or it satisfies the $t$-adic Littlewod conjecture.
\end{abstract}
\maketitle

\section{Introduction}

The classical $p$-adic Littlewood conjecture was first introduced by
de Mathan and Teulie~\cite{mat_teu_2004} in 2004. It states that for
any prime number $p$ and any real number $\alpha$ one has
$$
\liminf_{q\to\infty} q|q|_p\cdot ||q\alpha|| = 0.
$$
One can easily check that this equation is equivalent to
$$
\inf_{q\in \NN, k\ge 0} q\cdot ||qp^k\alpha||=0.
$$
In other words, this conjecture suggests that the numbers $p^k\alpha$ can not
be uniformly badly approximable for all $k\in \NN$, i.e. the set of partial
quotients of numbers $p^k\alpha$ is unbounded. For more details and a current
progress on this open problem we refer the reader to an
overview~\cite{bugeaud_2014}.

The $p$-adic Littlewood conjecture has an analogous formulation in function
fields. Let $\FF$ be a field and $\alpha\in \FF((t^{-1}))$ be the field of
formal Laurent series over $\FF$. We define an absolute value of $\alpha$ as
follows. For
\begin{equation}\label{def_alpha}
\alpha = \sum_{k=-h}^\infty a_k t^{-k};\quad a_{-h} \neq 0
\end{equation}
we set $|\alpha| := 2^h$. Instead of the absolute value we will sometimes use
a valuation of $\alpha$: $\nu(\alpha):= h$. The distance to the nearest
polynomial from $\alpha$ is defined as
$$
||\alpha||:= \min_{p\in \FF[t]} |\alpha - p| = \left|\sum_{k=1}^\infty a_k t^{-k}\right|.
$$
%The problem then is formulated as follows

\begin{problem}[$t$-adic Littlewood conjecture]
Given a field $\FF$, is it true that for any $\alpha\in \FF((t^{-1}))$,
$$
\inf_{q\in \FF[t]\setminus\{0\}, k\ge 0} |q|\cdot ||qt^k\alpha|| = 0?
$$
\end{problem}
Note that in the $t$-adic Littlewood conjecture ($t$-LC) the factor $t^k$
plays the role of $p^k$ from its $p$-adic counterpart. In more general
formulation of the conjecture, any polynomial $p(t)\in \FF[t]$ can be placed
instead of $t$, however in this article we restrict ourselves to narrower
Paroblem~A.

If the ground field $\FF$ is infinite, de Mathan and
\'Teulie~\cite{mat_teu_2004} proved that $t$-LC fails for some series
$\alpha$. Later, Bugeaud and de Mathan~\cite{bug_mat_2008} provided explicit
counterexamples. The case of the finite field $\FF$ appears to be much
harder, and until recently it was not clear if the $t$-adic Littlewood
conjecture holds or fails in that case. However Adiceam, Nesharim and
Lunnon~\cite{ad_ne_lu_2019} managed to construct an explicit counterexample
to $t$-LC for fields $\FF$ of characteristic 3. While the problem is still
open for finite fields of other characteristics, it is now generally believed
to be false for all fields $\FF$.

Given a ground field $\FF$, we denote by $\elc=\elc(\FF)$ the set of
exceptions to $t$-LC, namely
\begin{equation}\label{def_elc}
\elc:= \{\alpha\in \FF((t^{-1}))\;:\; \inf_{q\in
\FF[t]\setminus\{0\}, k\ge 0} |q|\cdot ||qt^k \alpha|| >0\}.
\end{equation}
For any $\delta>0$, let $\elc_\delta$ be a subset of $\elc$ where the
condition in~\eqref{def_elc} is replaced by
$$
\inf_{q\in \FF[t]\setminus\{0\}, k\ge 0} |q|\cdot ||qt^k \alpha|| \ge \delta.
$$
Note that $\elc$ (respectively $\elc_\delta$) belongs to the bigger set
$\bad$ (respectively $\bad_\delta$) of badly approximable series, where
$$
\bad:=\{\alpha\in \FF((t^{-1}))\;:\; \liminf_{q\in \FF[t]} |q|\cdot ||q \alpha|| >0\}\quad\mbox{and}
$$$$
\bad_\delta:=\{\alpha\in \FF((t^{-1}))\;:\; \liminf_{q\in \FF[t]} |q|\cdot
||q \alpha|| \ge \delta\}.
$$

Let's consider the case $\FF = \QQ$. We know that there are plenty
of exceptions to the $t$-LC in this case. For example, it is not
hard to construct such an exception by requiring that numerators (or
denominators) of the coefficients of $\alpha$ grow fast enough.
However it is usually challenging to verify $t$-LC for a given
series $\alpha$ which passes the growth conditions. The author is
aware of only two such results. Bugeaud and de
Mathan~\cite{bug_mat_2008} verified the conjecture for $\alpha =
e^{1/t}$. Adiceam, Nesharim and Lunnon gave a computer assisted
proof that a generating function of the Paper-Folding sequence lies
in $\elc_{2^{-4}}(\FF_3)$. The last result straightforwardly implies
that the same series defined over $\QQ$ also fails $t$-LC. On the
other hand, a number of results has been produced in the last
decades which verify a weaker condition for certain families of
series, including the generating function for the Thue-Morse
sequence and certain infinite products. Namely, they show that those
series belong to $\bad$. Some (but not all) papers of this flavour
are \cite{al_pe_we_we_1998, badziahin_2018, bu_ha_we_ya_2016,
coons_2013}.

In this paper we consider a solution $\alpha = g_P(t)$ of the
following Mahler functional equation:
\begin{equation}\label{def_gp}
g_P(t) = P(t) g_P(t^d);\quad P(t)\in \FF[t];\; d\in \ZZ_{\ge 2}.
\end{equation}
In~\cite{badziahin_2018}, several conditions on $P$ were provided
which guarantee that $g_P$ is badly approximable. One of the results
from there is
\begin{theorem}[Badziahin, 2019]\label{th1}
Let $P = x+u$ be linear function with $u\in \QQ$. Then $g_P$ belongs to
$\bad_{1/2}$ for all $u$ except $u=0$ and $u=1$. In the latter two cases
$g_P$ is a rational function.
\end{theorem}

Note that for $P= x-1$ the series $g_P$ is a generating function for
the Thue-Morse sequence over the alphabet $\{-1,1\}$. Hence we call
the functions $g_P$ from Theorem~\ref{th1} {\it generalised
Thue-Morse functions}. For convenience, we will use the notation
$g_u$ instead of $g_{x+u}$.

The main result of this note is the following strengthening of
Theorem~\ref{th1}:
\begin{theorem}\label{th5}
With the same polynomial $P=x+u$ as in Theorem~\ref{th1}, $g_P$
belongs to $\elc_{1/2}$ for all $u\in\QQ$, except $u=0,\pm 1$. For
$u=-1$, $g_P\in \bad_{1/2}\setminus\elc$.
\end{theorem}

Theorem~\ref{th5} states that the actual Thue-Morse function
$g_{-1}$ is the only generalised Thue-Morse function which satisfies
$t$-LC for $\FF=\QQ$.

For the sake of completeness, we also provide a similar result for
the case of finite fields~$\FF$.
\begin{theorem}\label{th4}
Let $\FF$ be a finite field and $P = x+u$ be a linear function with $u\in
\FF$. Then $g_P\not\in \elc$.
\end{theorem}

\section{Hankel determinants}

It is well known that $t$-LC has an equivalent reformulation in terms of
vanishing of Hankel determinants. Let $\alpha$ be given by~\eqref{def_alpha},
$n,l\in\ZZ$, $l\ge 0$. Define a two-dimensional sequence of Hankel matrices
as follows:
$$
H_\alpha(n,l):= \left(\begin{array}{cccc}
a_n& a_{n+1}&\cdots& a_{n+l}\\
a_{n+1}& a_{n+2}& \cdots&a_{n+l+1}\\
\vdots&\vdots&\ddots&\vdots\\
a_{n+l}&a_{n+l+1}&\cdots&a_{n+2l}
\end{array}\right).
$$
Here, by convention, $a_n=0$ for all $n<-h$. We follow the notation
from~\cite{ad_ne_lu_2019}.

\begin{definition}
Let $d\ge 2$. The series $\alpha\in \FF((t^{-1}))$ is said to have deficiency
$d$ if there exist integers $n\ge 1$ and $l\ge 0$ such that $d-1$ matrices
\begin{equation}\label{def1}
H_\alpha(n,l),H_\alpha(n,l+1),\ldots, H_\alpha(n, l+d-2)
\end{equation}
are singular but for any $n\ge 1, l\ge 0$ in the sequence of $d$ matrices of
the form
$$
H_\alpha(n,l),H_\alpha(n,l+1),\ldots, H_\alpha(n, l+d-1)
$$
at least one of them is non-singular. If, for any $n\ge 1,l\ge 0$, none of
the matrices $H_\alpha(n,l)$ is singular, the series $\alpha$ is said to have
deficiency 1. It is said to have unbounded deficiency if for any $d\ge 2$
there exist $n\ge 1$ and $l\ge 0$ such that all matrices in~\eqref{def1} are
singular.
\end{definition}

The deficiency of a series allows to verify $t$-LC.
\begin{theorem}[Theorem 2.2 in \cite{ad_ne_lu_2019}]\label{th3}
A series $\alpha\in \FF((t^{-1}))$ has finite deficiency $d\ge 1$ if and only
if $\alpha\in \elc_{2^{-d}}$.
\end{theorem}

Another helpful notion is the continued fraction of a Laurent series. Its
definition is analogous to that for real numbers. Every $\alpha\in
\FF((t^{-1}))$ can be written as
$$
\alpha(t) = b_0(t) + \frac{1}{b_1(t) + \frac{1}{b_2(t) + \cdots}} =:
b_0(t) + \bigk_{i=1}^\infty \frac{1}{b_i(t)}.
$$
where the partial quotients $b_0(t), b_1(t), \ldots,$ are polynomials in
$\FF[t]$. An overview of continued fractions over function fields can be
found in~\cite{poorten_1998}. We will use one result from there, which is an
analogue of the Legendre theorem.

\begin{theorem}\label{th2}
Let $\alpha\in \FF((t^{-1}))$. A reduced fraction $p/q\in \FF(t)$ is a
convergent of $\alpha$ if and only if
$$
\nu\left( \alpha - \frac{p}{q}\right) < -2\deg(q).
$$
Moreover, if $p_n(t)/q_n(t)$ is $n$th convergent of $\alpha$ then
$$
\nu\left( \alpha - \frac{p_n}{q_n}\right) = -\deg(q_n)-\deg(q_{n+1}) = -2\deg(q_n) - \deg(b_{n+1}).
$$
\end{theorem}
Here, a convergent of the series $\alpha$ is any finite continued fraction
formed by the first partial quotients of $\alpha$:
$$
\frac{p_n(t)}{q_n(t)} = [b_0(t);b_1(t),\ldots, b_n(t)].
$$

It is more convenient to renormalise partial quotients and write the
continued fraction of $\alpha$ as follows:
\begin{equation}\label{eq_cf}
\alpha(t) = b_0(t) + \bigk_{i=1}^\infty \frac{\beta_i}{b^*_i(t)},
\end{equation}
where $\beta_i\in \FF\setminus \{0\}$ and polynomials $b^*_i$ are monic
multiples of $b_i$. It appears that the parameters $\beta_i$ are closely
linked with Hankel determinants. Indeed, the next theorem appeared in its
full generality in the work of Han~\cite{han_2015} for Jakobi continued
fractions and later it was reformulated~\cite{badziahin_2018} for standard
continued fractions\footnote{There is a small mistake in the formulation of
Theorem~H2 in~\cite{badziahin_2018}: there must be $-\beta_2$, ... $-\beta_n$
instead of $\beta_2, \ldots, \beta_n$.}.

\begin{theorem}[Theorem H2 in~\cite{badziahin_2018}]\label{th_hankel2}
Let $\alpha\in\FF((t^{-1}))$ be such that its Hankel continued fraction is
given by~\eqref{eq_cf}. Denote by $s_n$ the degree of $q_n$ for $n$th
convergent of $\alpha$. Then, for all integers $n\ge 1$ all non-vanishing
Hankel determinants are given by
\begin{equation}\label{hankdet2}
H_\alpha(1,s_n-1) = (-1)^\epsilon
\beta_1^{s_n}(-\beta_2)^{s_n-s_1}\cdots (-\beta_{n})^{s_n-s_{n-1}}.
\end{equation}
where $\epsilon=\sum_{i=1}^{n} k_i(k_i-1)/2$ and $k_i = s_i-s_{i-1}$.
\end{theorem}

If all partial quotients of $\alpha$ are linear then we can say more.
\begin{theorem}[Pages 2,3 in \cite{han_2015a}]\label{th_hankel}
Let $\alpha\in\FF((t^{-1}))$ be such that its Hankel continued fraction is
given by~\eqref{eq_cf}. Assume that all partial quotients of $\alpha$ are
linear, i.e. $b^*_n(t) = t+\alpha_n$. Then all Hankel determinants
$H_\alpha(1,n)$ do not vanish and moreover
\begin{equation}\label{hankdet}
\det H_\alpha(1,n-1) = (-1)^{\frac{n(n-1)}{2}}\beta_1^{n}\beta_2^{n-1}\cdots \beta_{n}.
\end{equation}
Conversely, for $n\ge 3$ the coefficients $\alpha_n$ and $\beta_n$ can be
calculated from the Hankel determinants:
\begin{equation}\label{hankalp}
\alpha_n\! =\! \frac{-1}{\det H_\alpha(2,\!n\!-\!2)\!}\!\left(\!\!\frac{\det H_\alpha(1,\!n\!-\!2)\det H_\alpha(2,\!n\!-\!1)}{\det H_\alpha(1,n-1)}\! +\! \frac{\det H_\alpha(1,\!n\!-\!1)\det H_\alpha(2,\!n\!-\!3)}{\det H_\alpha(1,n-2)}\!\!\right)\!\!,
\end{equation}
\begin{equation}\label{hankbet}
\beta_n = -\frac{\det H_\alpha(1,n-3)\det H_\alpha(1,n-1)}{(\det H_\alpha(1,n-2))^2}.
\end{equation}
\end{theorem}

One of the important corollaries of Theorem~\ref{th_hankel2} is the following
result.
\begin{theorem}[Corollary~1 in~\cite{badziahin_2018}]\label{th8}
Let $\alpha\in\FF((t^{-1}))$. Then $H_\alpha(1,n-1)$ is non-singular if and
only if there exists a convergent $p/q$ of $\alpha$ with coprime $p$ and $q$
such that $\deg(q) = n$.
\end{theorem}

Theorem~\ref{th8} with Theorem~\ref{th2} imply that the deficiency of
$\alpha$ coincides with
$$
\max_{n\ge 1, j\ge 0} \{\deg(b_{n,j})\},
$$
where $b_{n,j}$ is the $n$th partial quotient of $t^j\alpha$.

\section{Properties of $g_P$}

One can check that the series $g_P$ from~\eqref{def_gp} can alternatively be
written as an infinite product:
\begin{equation}\label{def_gp2}
g_P(t) = t^{-1} \prod_{i=0}^\infty P^*(t^{-d^i}),\quad P^*(t) = t^dP(t^{-1}).
\end{equation}
With help of this formula we can compute its terms $a_i$. While it is
possible to do for an arbitrary value of $d$ we only need it for $d=2$.
\begin{proposition}\label{prop2}
Let $d=2$ and $P(t) = t+u$. Then the coefficient $a_n$ at $t^{-n}$ of $g_P$
equals $u^{\tau_2(n-1)}$, where $\tau_2(n)$ is the number of 1's in the
binary expansion of $n$.
\end{proposition}

\begin{proof}
After expanding the brackets in the infinite product~\eqref{def_gp2}, we have
$$
\prod_{i=0}^\infty (1+ut^{-2^i}) = \sum_{\vi}  u^{1(\vi)}\,t^{-\sum_{k=0}^\infty 2^{i_k}},
$$
where $\vi = (i_0,i_1,i_2,\ldots)$ runs through all element in $\{0,1\}^\NN$
such that the number of 1's in $\vi$ is finite, and $1(\vi)$ returns the
number of ones in $\vi$. Since every nonnegative integer can be written as a
sum of different powers of two in the unique way, the last sum can be
rewritten as
$$
\sum_{n=0}^\infty u^{\tau_2(n)} t^{-n}.
$$
Multiplying this series by $t^{-1}$ finishes the proof.
\end{proof}

A straightforward corollary of Proposition~\ref{prop2} is that $a_{2n} =
ua_{2n-1}$ and $a_{2n+1} = a_{n+1}$ for the coefficients of $g_P$. We will
use this fact in the proof of Theorem~\ref{th5}.

\begin{proposition}\label{prop1}
Let $g_P$ be given by~\eqref{def_gp}. Then $g_P \in \elc$ if and only if the
deficiency of $g_P$ is at most $\deg(P)/(d-1)$ or equivalently,
$$
g_P\in\elc_{2^{-\deg(P)/(d-1)}}.
$$
\end{proposition}

\begin{proof}
If the deficiency of  $g_P$ is at most $\deg(P)/(d-1)$ then, by
Theorem~\ref{th3}, $g_P$ is in $\elc$. Assume now that the deficiency of
$g_P$ is bigger than $\deg(P)/(d-1)$. Then, for some positive integer $j$,
there exists a partial quotient $b_{n,j}$ of $t^jg_p(t)$ with $\deg(b_{n,j})>
\deg(P)/(d-1)$. Suppose that $g_P$ is in $\elc$. That means that the degrees
of all partial quotients of $t^jg_p(t)$ for all $j\ge 0$ are bounded from
above by an absolute constant. Hence, without loss of generality, we may
assume that $b_{n,j} = \max\{b_{n',j'}: n'\ge 0, j'\ge 0\}$. In view of
Theorem~\ref{th2}, we have that
$$
\nu\left(t^jg_P(t) - \frac{p_{n-1}(t)}{q_{n-1}(t)}\right) = -2\deg(q_{n-1})-\deg(b_{n,j}),
$$
where $p_{n-1}/q_{n-1}$ is the $(n-1)$'th convergent of $t^j g_p(t)$. Then,
$$
\nu\left(t^{dj}g_P(t^d) - \frac{p_{n-1}(t^d)}{q_{n-1}(t^d)}\right) = -2\deg(q_{n-1}(t^d))-d\deg(b_{n,j}(t)),
$$
\begin{equation}\label{eq_def}
\stackrel{\eqref{def_gp}}{\Longrightarrow} \nu\left(t^{dj}g_P(t) -
\frac{P(t)p_{n-1}(t^d)}{q_{n-1}(t^d)}\right) =
-2\deg(q_{n-1}(t^d))-(d\deg(b_{n,j}(t))-\deg(P(t))).
\end{equation}
Note that $d\deg(b_{n,j})-\deg(P) > \deg(b_{n,j})$ and hence, by
Theorem~\ref{th2},
$$
\frac{P(t)p_{n-1}(t^d)}{q_{n-1}(t^d)}
$$
is $m$'th convergent of $t^{dj}g_P(t)$ for some $m$. Moreover,
$\deg(b_{m+1,dj})> \deg(b_{n,j})$ and that contradicts the assumption that
$b_{n,j}$ is of maximal degree. \end{proof}

\section{Continued fraction of $t^n g_u$}

Throughout this section we let $d=2$ and $P(t) = x+u$ for some parameter
$u\in \FF$.
%In view of Proposition~\ref{prop1}, $t^ng_u(t)$ is badly
%approximable for some $n\in\ZZ_{\ge 0}$ if and only if all Hankel
%determinants $H_{g_u}(n+1,l)$ are non-singular. The last statement is in turn
%equivalent to that all partial quotients of $t^n g_u(t)$ are linear.
In~\cite{badziahin_2018} the recurrent formula for partial quotients of the
continued fraction of $g_u$ is provided.

\begin{theorem}[Theorem~1 in \cite{badziahin_2018}]\label{th_recur}
Let $u\in\FF$. If $g_u(t)$ is badly approximable then its continued
fraction is
$$
g_u(x) = \bigk_{i=1}^\infty \frac{\beta_i}{x+\alpha_i}
$$
where the coefficients $\alpha_i$ and $\beta_i$ are computed by the
formula
\begin{equation}\label{recur_d2}
\begin{array}{l}
\alpha_{2k+1} = -u,\; \alpha_{2k+2}=u;\\
\displaystyle\beta_1 = 1,\; \beta_2 = u^2-u,\; \beta_{2k+3} =
-\frac{\beta_{k+2}}{\beta_{2k+2}},\; \beta_{2k+4} = \alpha_{k+2}+u^2
- \beta_{2k+3}
\end{array}
\end{equation}
for any $k\in\ZZ_{\ge 0}$.
\end{theorem}

In the same paper the author shows that for $\FF=\QQ$ and for any $u\in
\QQ\setminus \{0,1\}$ the series $g_u$ is indeed badly approximable.
Moreover, for every $m\in\NN$, the parameter $\beta_m$ can be written as a
rational function of $u$
$$
\beta_m = \frac{e_m(u)}{d_m(u)},
$$
where the leading and the smallest non-zero coefficients of both polynomials
$e_m$ and $d_m$ are $\pm 1$. Now we compute the valuation  of the values
$\beta_m$ as rational functions of $m$. By $\nu_2(n)$ we denote the 2-adic
valuation of an integer value $n$.

\begin{proposition}\label{prop3}
For all $m\ge 2$ one has $\nu(\beta_m) = 2(1+\nu_2(m-1))$.
\end{proposition}

\begin{proof}
We prove it by induction. For $m=2$, $\nu(\beta_2) = \deg(u^2-u) = 2 =
2(1+\nu_2(1))$. Assume the statement is true for all values
$\beta_2,\beta_3,\ldots, \beta_{m-1}$ and prove it for $\beta_m$. In
particular, the assumption implies that the valuations of all $\beta_i$'s for
odd values of $i<m$ are strictly less than 2.

If $m$ is even, $m = 2m_1$, then by~\eqref{recur_d2}, $\beta_{2m_1} = u^2
+\alpha_{m_1}-\beta_{2m_1-1}$. One can easily compute that the valuation of
the last expression is 2, which is $2(1+\nu_2(m-1))$.

If $m$ is odd, $m=2m_1+1$, then $\beta_{2m_1+1} = -\beta_{m_1+1}/
\beta_{2m_1}$ and
$$
\nu(\beta_{2m_1+1}) = 2(1+\nu_2(m_1) -1-\nu_2(2m_1-1)) = 2(1+ \nu_2(2m_1)).
$$
\end{proof}

This proposition together with Theorem~\ref{th_hankel} allows us to compute
the degrees of the determinants of the Hankel matrices $H_{g_u}(1,n)$.

\begin{proposition}\label{prop4}
$\deg H_{g_u}(1,n) = 2\sigma(n)$, where the function $\sigma$ is defined as
the sum:
$$
\sigma(n) = \sum_{i=1}^n \tau_2(i).
$$
\end{proposition}

\begin{proof}
We compute
$$
\nu \left(\prod_{i=1}^n \beta_i\right) = 2\sum_{i=2}^n (1+\nu_2(i-1)) = 2\tau_2(n-1).
$$
For the last equality we use the formula $\tau_2(n) = \tau_2(n-1) +
(1+\nu_2(n))$, which is straightforwardly verified. Then the statement of the
proposition directly follows from~\eqref{hankdet}.
\end{proof}

%Now we consider the function $t^n g_u(t)$.
%The equation~\eqref{hankdet} from Theorem~\ref{th_hankel} can be rewritten as
%\begin{equation}\label{hankdetn}
%\det H_\alpha(n+1,m-1) = (-1)^{\frac{m(m-1)}{2}}\beta_{n,1}^{m}\beta_{n,2}^{m-1}\cdots \beta_{n,m}.
%\end{equation}

We finish this section by showing that similar relations to \eqref{recur_d2}
hold between some partial quotients of $t^n g_u(t)$ and $t^{2n} g_u(t)$. We
need this result for Theorem~\ref{th5}, however it may be of independent
interest. If all partial quotients of $t^n g_u(t)$, except the zero's one,
are linear then we can write
$$
t^ng_u(t) = b_{n,0}(t) + \bigk_{i=1}^\infty
\frac{\beta_{n,i}}{t+\alpha_{n,i}}.
$$

\begin{theorem}\label{th6}
Let $u\in\FF$ and $n\in\ZZ_{\ge 0}$. If all partial quotients of $t^ng_u(t)$
and $t^{2n}g_u(t)$ are linear then the coefficients $\alpha_{2n,i}$ and
$\beta_{2n,i}$ of the continued fraction of $t^{2n}g_u(t)$ are computed by
the formula
\begin{equation}\label{nrecur_d2}
\begin{array}{l}
\alpha_{2n,2k+1} = -u,\; \alpha_{2n,2k+2}=u;\;\beta_{2n,1} = u^{\tau_2(2n)},\; \beta_{2n,2} = u^2 -
u^{\tau_2(n+1)-\tau_2(n)},\\
\displaystyle\beta_{2n,2k+3} =
-\frac{\beta_{n,k+2}}{\beta_{2n,2k+2}},\; \beta_{2n,2k+4} =
\alpha_{n,k+2}+u^2 - \beta_{2n,2k+3}
\end{array}
\end{equation}
for any $k\in\ZZ_{\ge 0}$.

Moreover, if all partial quotients of $t^ng_u(t)$ are linear and all
values $\beta_{2n,1}, \ldots, \beta_{2n,k}$ computed by
formulae~\eqref{nrecur_d2}, are nonzero then the partial quotients
$b_{2n,1}(t), \ldots, b_{2n,k}(t)$ of $t^{2n}g_u(t)$ are linear.
\end{theorem}

The proof is similar to that of Theorem~\ref{th_recur}
in~\cite{badziahin_2018}. Therefore we may omit some details and concentrate
on the idea of the proof.

{\bf Remark.} Since we always have $\tau_2(n+1)\le \tau_2(n)+1$,
Theorem~\ref{th6} implies that $\beta_{2n,2}$ is always a non-zero rational
function of $u$. Moreover, the maximal and minimal non-zero coefficients of
its numerator and denominator equal $\pm 1$.

\begin{proof}
We have that for any $m\in\ZZ_{\ge 0}$,
$$
t^{m} g_u(t) = b_{m,0}(t) + a_{m+1}t^{-1} + a_{m+2}t^{-2}+\cdots,
$$
where $b_{m,0}$ is a polynomial. One can easily check that the first
convergent of $t^{m} g_u(t)$ is
\begin{equation}\label{eq_bam1}
\frac{a_{m+1}}{t-a_{m+2}/a_{m+1}} = \frac{\beta_{m,1}}{t+\alpha_{m,1}}.
\end{equation}
For $m=2n$ we have $a_{m+1} = \tau_2(2n)$ and $a_{m+2} = ua_{m+1}$, hence
$\beta_{2n,1} = u^{\tau_2(2n)}$ and $\alpha_{2n,1} = -u$.

In the proof of Proposition~\ref{prop1} we showed that if
$p_{n,k}(t)/q_{n,k}(t)$ is $k$th convergent of $t^ng_u(t)$ then
\begin{equation}\label{eq1}
(t+u)p_{n,k}(t^2)/q_{n,k}(t^2) = p_{2n,2k}(t)/q_{2n,2k}(t)
\end{equation}
is $2k$'th convergent of $t^{2n}g_u(t)$. Moreover, the fraction on
the left hand side is reduced because otherwise the partial quotient
$b_{2n,2k}(t)$ is not linear. From the theory of continued fractions
we know that the convergents $p_{n,k}/q_{n,k}, p_{n,k+1}/q_{n,k+1}$
and $p_{n,k+2}/q_{n,k+2}$ are linked by the following formulae:
\begin{equation}\label{eq_pqn}
\begin{array}{l}
q_{n,k+2} = (t+\alpha_{n,k+2})q_{n,k+1} + \beta_{n,k+2}q_{n,k};\\
p_{n,k+2} = (t+\alpha_{n,k+2})p_{n,k+1} + \beta_{n,k+2}p_{n,k}.
\end{array}
\end{equation}

For $k=1$ equations~\eqref{eq_bam1} and~\eqref{eq1} imply that $q_{2n,2}(t) =
t^2 - \frac{a_{n+2}}{a_{n+1}} = t^2 - u^{\tau_2(n+1)-\tau_2(n)}$.
By~\eqref{eq_pqn}, this is equal to
$$
t^2 - u^{\tau_2(n+1)-\tau_2(n)} = (t+\alpha_{2n,2})(t-u) + \beta_{2n,2}
$$
and hence $\alpha_{2n,2} = u, \beta_{2n,2} = u^2 -
u^{\tau_2(n+1)-\tau_2(n)}$.

Fix any $k\in\NN$. From~\eqref{eq1} we can see that $p_{2n,2k} \equiv
p_{2n,2k+2} \equiv 0$ (mod $t+u$). Thus an application of~\eqref{eq_pqn} for
$p_{2n,2k+2}$ modulo $t+u$ gives
$$
0\equiv (t+\alpha_{2n,2k+2}) p_{2n,2k+1} + 0 \pmod{t+u}.
$$
Since $p_{2n,2k+1}$ and $p_{2n,2k+2}$ are coprime, $t+\alpha_{2n,2k+2}$ is a
multiple of $t+u$ and hence $\alpha_{2n,2k+2} = u$.

Next, we use formulae~\eqref{eq1} and~\eqref{eq_pqn} to provide two different
relations between $q_{2n,2k}(t)$, $q_{2n,2k+2}(t)$ and $q_{2n,2k+4}(t)$.
\begin{equation}\label{eq2}
q_{2n,2k+4} = (t^2+\alpha_{n,k+2}) q_{2n,2k+2} +
\beta_{n,k+2}q_{2n,2k}.
\end{equation}
Alternatively, from~\eqref{eq_pqn} for $q_{2n,2k+2}$ we derive that
$$
q_{2n,2k+1} = \frac{q_{2n,2k+2} - \beta_{2n,2k+2}q_{2n,2k}}{t+u}
$$
This equation together with~\eqref{eq_pqn} for $q_{2n,2k+3}$ and
$q_{2n,2k+4}$ yields the relation:
$$
q_{2n,2k+4} = ((t+\alpha_{2n,2k+3})(t+u) + \beta_{2n,2k+4}+\beta_{2n,2k+3})q_{2n,2k+2} - \beta_{2n,2k+3}\beta_{2n,2k+2}q_{2n,2k}.
$$
We compare the last equation with~\eqref{eq2} to get that
$$
t^2+\alpha_{n,k+2} = (t+\alpha_{2n,2k+3})(t+u) + \beta_{2n,2k+3} + \beta_{2n,2k+4}\quad \mbox{and}
$$$$
\beta_{n,k+2} = -\beta_{2n,2k+2}\beta_{2n,2k+3}.
$$
Then the conditions~\eqref{nrecur_d2} follow immediately.

Now we prove the second claim of the theorem. Observe that for
$u\neq 0$, $\beta_{2n,1} = u^{\tau_2(2n)}$ is never zero. Notice
that, by~\eqref{eq1}, the second partial quotient $b_{2n,2}$ is not
linear if $q_{n,1}(t^2)$ is a multiple of $t+u$ or equivalently,
$t-u^2\mid q_{n,1}(t)$. Finally, the formulae~\eqref{eq_bam1} and
\eqref{nrecur_d2} imply that in this case $\beta_{2n,2}=0$.

Take the smallest $k$ such that $b_{2n,k}$ is not linear (if such $k$ does
not exist then there is nothing to prove). We have already shown that $k>2$.
Notice also that $k$ can not be odd. Indeed, in that case
$p_{2n,k-1}/q_{2n,k-1}$ is computed by~\eqref{eq1} where by assumption, the
fraction on the left hand side of~\eqref{eq1} is in its reduced form.
Then~\eqref{eq_def} implies that $b_{2n,k}$ is linear.

For convenience, from now on we write $2k$ instead of $k$. If $b_{2n,2k}$ is
not linear then by~\eqref{eq1}, $t+u$ divides $q_{n,k}(t^2)$ and therefore
$$
q_{2n,2k-1}(t) = \frac{q_{n,k}(t^2)}{t+u}.
$$
We now combine this equation together with the following equations
$$
q_{2n,2k-2}(t) = q_{n,k-1}(t^2),\quad q_{2n,2k-4}(t) =
q_{n,k-2}(t^2),
$$$$
q_{2n,2k-3}(t) = \frac{q_{2n,2k-2} -
\beta_{2n,2k-2}q_{2n,2k-4}}{t+u},
$$
$$
q_{2n,2k-1}(t) = (t-u)q_{2n,2k-2}(t) + \beta_{2n,2k-1}q_{2n,2k-3}(t)
$$
and
$$
q_{n,k}(t^2) = (t^2+\alpha_{n,k})q_{n,k-1}(t^2) + \beta_{n,k}q_{n,k-2}(t^2)
$$
to get the equation
$$
\begin{array}{rl}
&(t^2+\alpha_{n,k})q_{n,k-1}(t^2) + \beta_{n,k}q_{n,k-2}(t^2)\\[1ex]
= &(t^2 - u^2 + \beta_{2n,2k-1})q_{n,k-1}(t^2)+
\beta_{2n,2k-1}\beta_{2n,2k-2}q_{n,k-2}(t^2).
\end{array}
$$
This equation implies that $\alpha_{n,k} = -u^2+\beta_{2n,2k-1}$ and
hence the value $\beta_{2n,2k}$ from the formula~\eqref{nrecur_d2}
equals zero.
\end{proof}

\section{Some relations between Hankel determinants}

Consider the Hankel matrix $H_{g_u}(n,l) = (a_{n+i+j})_{0\le i,j\le l}$. For
simplicity, from now on we will omit the index $g_u$ and call this matrix
$H(n,l)$. For any $0\le j\le \lfloor \frac{l-1}{2}\rfloor$ multiply
$(2j+1)$'th row by $u$ and subtract it from $(2j+2)$'th row. Then, since
$a_{2m}= ua_{2m-1}$, $(2j+2)$'th row becomes
$$
(0,\; a_{n+2j+2} - ua_{n+2j+1},\;0,\; a_{n+2j+4} -
ua_{n+2j+3},\;0,\;\ldots)
$$
for odd $n$ and
$$
(a_{n+2j+1} - ua_{n+2j},\;0,\; a_{n+2j+3} -
ua_{n+2j+2},\;0,\;\ldots)
$$
for even $n$. Reorder the rows as follows: $2,4,\ldots,
2\lfloor\frac{l+1}{2}\rfloor, 1, 3, \ldots, 2\lfloor\frac{l}{2}\rfloor+1$.
Then reorder the columns so that all zeroes from the first row are placed on
the right side of the matrix. Such transformations do not change the absolute
value of $\det (H(n,l))$. If $l$ is odd or $n$ is odd the determinant of the
resulting matrix can be computed as a product of the determinants of two
blocks:
$$
\begin{array}{rl}
|\det(H(n,l))| =& \big|\det (a_{2\lfloor\frac{n}{2}\rfloor +1 + 2i+2j})_{0\le
i,j\le \lfloor\frac{l}{2}\rfloor}\\[1ex]
&\cdot \det( a_{2\lceil \frac{n}{2}\rceil +1
+2i+2j} - ua_{2\lceil\frac{n}{2}\rceil +2i+2j})_{0\le i,j\le
\lfloor\frac{l-1}{2}\rfloor}\big|.
\end{array}
$$
Since $a_{2m+1} = a_{m+1}$ for all $m\in\NN$, the first matrix in the product
is $H(\lfloor\frac{n}{2}\rfloor +1,\lfloor\frac{l}{2}\rfloor)$. For the
second matrix we additionally use the formula $a_{2m} = ua_{2m-1}$ to get
$$
(a_{\lceil\frac{n}{2}\rceil +1+i+j} - u^2 a_{\lceil\frac{n}{2}\rceil
+i+j})_{0\le i,j\le \lfloor\frac{l-1}{2}\rfloor} =: \tilde{H}
\left(\left\lceil\frac{n}{2}\right\rceil,
\left\lfloor\frac{l-1}{2}\right\rfloor\right).
$$
Finally, if at least one of $n$ and $l$ is odd we derive the equation
\begin{equation}\label{eq_hnl}
|\det (H(n,l))| = \left|\det
H\left(\left\lfloor\frac{n}{2}\right\rfloor
+1,\left\lfloor\frac{l}{2}\right\rfloor\right)\right|\cdot
\left|\det \tilde{H} \left(\left\lceil\frac{n}{2}\right\rceil,
\left\lfloor\frac{l-1}{2}\right\rfloor\right)\right|.
\end{equation}

If both $n$ and $l$ are even we compute the determinant of $H(n,l)$ in a
slightly different way. For any $1\le j\le l/2$ multiply $2j$'th row by $u$
and subtract it from $(2j+1)$'th row. Reorder the rows as follows:
$3,5,7,\ldots, l+1,1, 2,4,6,\ldots, l$. Then reorder the columns so that all
zeroes of the first row are placed in the right side of the matrix. Then, as
before, the determinant of the resulting matrix is the product of the
determinants of two blocks. For the first one we have
$$
\left(\begin{array}{ccccc} a_n&a_{n+2}&a_{n+4}&\cdots & a_{n+l}\\
a_{n+1}&a_{n+3}&\cdots&\cdots&a_{n+l+1}\\
a_{n+3}&a_{n+5}&\cdots&\cdots&a_{n+l+3}\\
\vdots&\vdots&\vdots&\ddots&\vdots\\
a_{n+l-1}&a_{n+l+1}&\cdots&\cdots&a_{n+2l-1}
\end{array}\right).
$$
By dividing each term of the first row by $u$ we get the row $(a_{n-1}\;
a_{n+1}\; \cdots\; a_{n+l-1})$. Finally, since each term of the matrix now
has an odd index we can use the equation $q_{2m+1} = a_{m+1}$ to get that the
determinant of the first block equals
$$
u\cdot \det (a_{\frac{n}{2} + i + j})_{0\le i,j\le \frac{l}{2}} =
u\cdot \det H\left(\frac{n}{2}, \frac{l}{2}\right).
$$

For the second block we have
$$
(a_{n+3+2i+2j} - ua_{n+2+2i+2j})_{0\le i,j\le \frac{l}{2}-1} =
(a_{\frac{n}{2}+2+i+j} - u^2 a_{\frac{n}{2}+1+i+j})_{0\le i,j\le
\frac{l}{2}-1},
$$
which is $\tilde{H}(\frac{n}{2}+1, \frac{l}{2}-1)$. Finally, for even $n$ and
$l$ we get the formula
\begin{equation}\label{eq_hnl2}
|\det H(n,l)| = |u|\cdot \left|\det H
\left(\frac{n}{2},\frac{l}{2}\right) \right|\cdot \left|\tilde{H}
\left(\frac{n}{2}+1, \frac{l}{2}-1\right)\right|.
\end{equation}

\section{Doubly monic rational functions}

We call a rational function $\alpha\in \QQ(t)$ doubly monic if it can be
written as a fraction $e/d$ such that $e,d\in \ZZ[t]$, $\gcd(e,d) =
\mathrm{const}$ and both the leading and the smallest non-zero coefficients
of both $e$ and $d$ are $\pm 1$.

In this paper we use two important basic properties of doubly monic functions
which are readily verified:
\begin{itemize}
\item[\bf{P1.}] If $\alpha_1,\alpha_2$ are doubly monic then so are
    $\alpha_1\alpha_2$ and $\alpha_1/\alpha_2$.
\item[\bf{P2.}] Let $\alpha$ be doubly monic. Then for all
    $t\in\QQ\setminus\{-1,0,1\}$ one has $0<|\alpha(t)|<\infty$.
\end{itemize}

Property P2 is in fact true for any algebraic number $t$ which is either not
integer or not a unit in the corresponding number field.

An easy implication of Property~P1 and Theorem~\ref{th_hankel} is the
following fact:

\begin{lemma}\label{lem1}
Assume that for some $n\in\ZZ_{\ge 0}$ the partial quotients
$b_{n,1},b_{n,2},\ldots, b_{n,m-1}$ of $t^n g_u(t)$ are linear and the
corresponding parameters $\beta_{n,1}, \beta_{n,2},\ldots, \beta_{n,m-1}$ are
all doubly monic. Then $H_{g_u}(n+1,m-1)$ is doubly monic if and only if
$\beta_{n,m}$ is doubly monic.
\end{lemma}

\section{Proof of Theorem~\ref{th5}}

In view of Proposition~\ref{prop1}, for $d=2$ and linear $P$, $g_P$ belongs
to $\elc$ if and only if all Hankel determinants $H_{g_P}(n,l)$ are
non-singular. Then with help of Proposition~\ref{prop2} for $P(t) = t-1$ we
straightforwardly get $a_3=-1, a_4=1, a_5=-1$ and hence
$$
\det(H_{g_{-1}}(3,1)) = \left|\begin{array}{cc}
-1&1\\
1&-1
\end{array}
 \right| = 0,
 $$
and thus $g_{-1} \not \in \elc$.

The remaining problem is to show that for any rational $u\neq 0,\pm1$ and for
all $n\ge 0, l\ge 0$ one has $\det(H(n,l))\neq 0$. Note that each coefficient
$a_n$ of $g_{u}$ is a polynomial (in fact, monomial) in $u$ and hence
$\det(H(n,l))$ are also polynomials. The main idea of the proof is to show
that these polynomials are all doubly monic. That would immediately imply
that the only possible rational roots of such polynomials are $u=0$ and
$u=\pm 1$.

We start by showing that $\det H(n,l)$ are monic polynomials. Moreover, we
prove that
\begin{equation}\label{eq_degh}
\deg(\det H(n,l)) = \sigma(l) + \sigma(n+l-1) - \sigma(n-2)\quad\mbox{and}
\end{equation}
\begin{equation}\label{eq_degth}
\deg(\det \tilde{H}(n,l)) = 2(l+1) + \sigma(l) + \sigma(n+l-1) - \sigma(n-2).
\end{equation}
Here, by convention, we say that $\sigma(n)=0$ for all $n\le 0$.

 We use double induction. We start with the
induction on $n$ and within each inductional step we do the induction on $l$.

{\bf The base of induction.} For $n=1$, it is proven in~\cite[Lemma
8]{badziahin_2018} that $\det H(1,l)$ is monic. Indeed, it is shown there
that all the parameters $\beta_{0,l}$ for $l\in \ZZ_{\ge 0}$ are doubly monic
and hence Lemma~\ref{lem1} implies that so are $\det H(1,l)$.
Next,~\eqref{eq_degh} is an immediate implication of Proposition~\ref{prop4}.
Finally, with help of~\eqref{eq_hnl} we get that $\tilde H(1,l)$ is doubly
monic and
\begin{equation}\label{eq3}
\deg (\det\tilde H(1,l)) = \deg (\det H(1,2l+1)) - \deg (\det H(1,l)) =
2\sum_{i=l+1}^{2l+1} \tau_2(i).
\end{equation}
There is a bijection between two sets $\{l+1,l+2,\ldots, 2l+1\}$ and
$\{0,1,\ldots, l\}$ given by $\phi\;:\; m \mapsto \frac12(m\cdot |m|_2 -1)$.
Moreover, one can check that $\tau_2(m) = \tau_2(\phi(m)) +1$ and therefore
$$
\deg (\det\tilde H(1,l)) = 2\left(l+1 + \sum_{i=0}^l \tau_2(i) \right) = 2(l+1) + 2\sigma(l)
$$
and~\eqref{eq_degth} is verified for $n=1$.

For $l=0$ we obviously have $\det H(n,0) = a_n = u^{\tau_2(n-1)}$ which is
doubly monic and satisfies~\eqref{eq_degh}. Similarly,~\eqref{eq_degth} is
verified for $\tilde{H}(n,0) = a_{n+1} - u^2a_n$. A straightforward
computation also gives $\det H(n,1) = a_na_{n+2} - a_{n+1}^2$.

Let $|n|_2 <1$ . Then $a_n = a_{n+1} / u^{1+\nu_2(n)}$ and $a_{n+2} =
ua_{n+1}$. Hence
$$
\det H(n,1) = a^2_{n+1} (u^{-\nu_2(n)} - 1) = u^{2\tau_2(n)}(u^{-\nu_2(n)} - 1)
$$
is monic. Then~\eqref{eq_degh} is verified by the equation $2\tau_2(n) -
\nu_2(n) = \tau_2(1) + \tau_2(n-1) + \tau_2(n)$.

Let $|n+1|_2 <1$. Then $a_{n+2} = u^{1+\nu_2(n+1)} a_{n+1}$ and $a_n =
a_{n+1}/u$. Therefore
$$
\det H_{g_u}(n,1) = a^2_{n+1} (u^{\nu_2(n+1)} - 1) = u^{2\tau_2(n)}(u^{\nu_2(n+1)} - 1)
$$
which is also monic and~\eqref{eq_degh} is directly verified.

For $\tilde{H}(n,1)$, we have
$$
\det \tilde{H}(n,1) = (a_{n+1}-u^2a_n)(a_{n+3} - u^2 a_{n+2}) - (a_{n+2}-u^2 a_{n+1})^2.
$$
Since for all $m\in\ZZ_{\ge 0}$, $\nu(a_{m+1})\le \nu(a_m)+1$, we have that
the dominating term in the determinant above is $u^4(a_na_{n+2} - a_{n+1}^2)
= u^4 \deg H(n,1)$. This implies that $\tilde{H} (n,1)$ is monic and
satisfies~\eqref{eq_degth}.

{\bf Inductional step.} Assume that for all $l\in \ZZ_{\ge 0}$ the
determinants of $H(1,l) , \ldots, H(2n-1,l)$ are monic. We will prove that
$\det H(2n,l)$ and $\det H(2n+1,l)$ are also monic for all $l\in \ZZ_{\ge 0}$
and their degrees satisfy~\eqref{eq_degh}.

Consider~\eqref{eq_hnl} for $H(2n-1, 2l+1)$ where $l$ is arbitrary.
$$
\det H(2n-1, 2l+1) = \pm \det H(n,l) \cdot \det \tilde{H} (n, l).
$$
By assumption, both $\det H(2n-1, 2l+1)$ and $\det H(n,l)$ are monic and thus
so is $\det \tilde{H} (n,l)$. Moreover,~\eqref{eq_degh} implies
\begin{equation}\label{eq_degth2}
\deg(\det \tilde{H}(n,l)) = \sum_{i=l+1}^{2l+1} \tau_2(i) + \sum_{i=n+l}^{2n+2l-1} \tau_2(i) - \sum_{i=n-1}^{2n-3} \tau_2(i).
\end{equation}
By the same arguments as in the computation of~\eqref{eq3}, we get that this
expression equals
$$
(l+1) + (n+l) - (n-1) + \sum_{i=0}^l\tau_2(i) +\sum_{i=0}^{n+l-1}\tau_2(i) - \sum_{i=0}^{n-2} \tau_2(i)
$$$$
=2(l+1) + \sigma(l) + \sigma(n+l-1) - \sigma(n-2).
$$
This verifies~\eqref{eq_degth} for $\tilde{H}(n,l)$ where $n$ is fixed and $l$ is arbitrary.

According to the base of induction, $\det H(2n,1), \det H(2n,2), \det
H(2n+1,1)$ and $\det H(2n+1,2)$ are monic. Now assume that the determinants
of $H(i,j)$ for all $(i,j)\in \{2n,2n+1\}\times\{1,\ldots, l\}$ are monic and
show that so are $\det H(2n,l+1)$ and $\det H(2n+1,l+1)$.

Suppose that {\bf $l$ is even}. We apply~\eqref{eq_hnl} to get
$$
\det H(2n, l+1) = \pm \det H(n+1,l/2) \cdot \det \tilde{H} (n, l/2).
$$
Both $\det H(n+1,l/2)$ and $\det \tilde{H} (n, l/2)$ are monic, therefore
$\det H(2n, l+1)$ is monic. We also have
$$
\deg (\det H(2n,l+1)) = \sigma(l/2) + \sigma (n+l/2) - \sigma(n-1) + \sum_{i=l/2+1}^{l+1}
 \tau_2(i) + \sum_{i=n+l/2}^{2n+l-1}\tau_2(i) - \sum_{i=n-1}^{2n-3} \tau_2(i)
 $$
 $$
= \sigma(l+1) + \sigma(2n+l-1) +\tau_2(n+l/2) - \sigma (2n-3) - \tau_2(n-1).
$$
Now,~\eqref{eq_degh} follows from $\tau_2(n+l/2) = \tau_2(2n+l)$ and $\tau_2(2n-2) = \tau_2(n-1)$.

Consider $H(2n+1,l+1)$. For $l=2$, we have
$$
\det H(2n+1,3) = \pm \det H(n+1,1)\cdot \tilde{H} (n+1,1).
$$
From the base of induction we get that $H(2n+1,3)$ is monic and its degree is
$$
\sigma(1) + \sigma(n+1) - \sigma(n-1) + \sum_{i=2}^3 \tau_2(i) + \sum_{i=n+2}^{2n+3} \tau_2(i) - \sum_{i=n}^{2n-1} \tau_2(i)
$$
and~\eqref{eq_degh} is straightforwardly verified.

For the case $l\ge 4$ we use Theorem~\ref{th6}. It implies that
$\beta_{2n,l+2} = u^2 + \alpha_{n,l/2+1} - \beta_{2n,l+1}$. Then we apply
Theorem~\ref{th_hankel} to $t^{2n}g_u(t)$ and $t^ng_u(t)$. For $l\ge 4$
equations~\eqref{hankalp} and~\eqref{hankbet} provide the estimates for the
valuations of $\alpha_{n,l/2+1}$ and $\beta_{2n,l+1}$ as rational functions
of $u$:
$$
\nu(\beta_{2n,l+1}) = \nu(\det H(2n+1,l)) + \nu(\det H(2n+1,l-2)) - 2\nu(\det H(2n+1,l-1)),
$$
$$
 = \tau_2(l) - \tau_2(l-1) + \tau_2(2n+l)- \tau_2(2n+l-1)\le 0
$$
and
$$
\begin{array}{rl}
\nu(\alpha_{n,l/2+1}) \le &\!\!\!\!\max\{ \nu(\det H(n+1, l/2-1)) + \nu(\det H(n+2, l/2)) - \nu(\det H(n+1, l/2)),\\
&\!\!\!\!\nu(\det H(n+1, l/2)) + \nu(\det H(n+2,l/2-2))-\nu(\det H(n+1,l/2-1))\}\\
-&\!\!\!\!\nu(\det H(n+2,l/2-1))\\
=&\!\!\!\! \max\{ \tau_2(n+l/2+1) - \tau_2(n+l/2), \tau_2(l/2) - \tau_2(l/2-1)\} \le 1.
\end{array}
$$
These two estimates imply that in the expression $u^2 +
\alpha_{n,l/2+1} - \beta_{2n,l+1}$ the term $u^2$ is dominating,
therefore $\beta_{2n,l+2}$ is monic and thus, by the second
statement of Theorem~\ref{th6}, $\det H(2n+1,l+1)$ is also monic.
For the degree of the determinant, $\nu(\beta_{2n,l+2})=2$ and from
Theorem~\ref{th_hankel} we have
$$
\deg(\det H(2n+1,l+1)) = 2 \deg(\det H(2n+1,l)) + \nu(\beta_{2n,l+2}) - \deg(\det H(2n+1,l-1))
$$$$
 = \sigma(2n-1) + \sigma(l) + \tau_2(l) + \sigma(2n+l) + \tau_2(2n+l) + 2.
$$
Since $l$ is even, one has $\tau_2(l) + 1 = \tau_2(l+1), \tau_2(2n+l) + 1 = \tau_2(2n+l+1)$ and then~\eqref{eq_degh} clearly follows.

Suppose that {\bf $l$ is odd}. We apply~\eqref{eq_hnl} to the matrix
$H(2n+1,l)$ and get
$$
\det H(2n+1,l) = \pm \det H(n+1,(l-1)/2)\cdot \det \tilde{H}(n+1, (l-1)/2).
$$
By inductional assumption, both $\det H(2n+1,l)$ and $\det H(n+1,(l-1)/2)$
are monic and hence so is $\det \tilde{H}(n+1, (l-1)/2)$. Moreover,
$$
\deg(\det\tilde{H}(n+1, (l-1)/2)) = \sum_{i=(l+1)/2}^l\tau_2(i)+ \sum_{i=n+(l+1)/2}^{2n+l} \tau_2(i) - \sum_{i=n}^{2n-1} \tau_2(i)
$$$$
= l+1 + \sigma((l-1)/2) + \sigma(n+(l-1)/2) - \sigma(n-1)
$$
and~\eqref{eq_degth} follows.

We apply~\eqref{eq_hnl2} to get
$$
\det H(2n, l+1) = \pm u \det H(n,(l+1)/2) \cdot \det \tilde{H} (n+1, (l-1)/2).
$$
As we have already shown, $\det H(n,(l+1)/2)$ and $\det \tilde{H} (n+1,
(l-1)/2)$ are monic, therefore $\det H(2n, l+1)$ is also monic. Moreover, we
have
$$
\begin{array}{rl}
\deg (\det H(2n,l+1)) =& 1 + \sigma((l+1)/2) + \sigma(n+(l-1)/2) - \sigma
(n-2) \\[1ex]
+&\displaystyle \sum_{i=(l+1)/2}^{l}\tau_2(i)
 \sum_{i=n+(l+1)/2}^{2n+l}\tau_2(i) - \sum_{i=n}^{2n-1} \tau_2(i)
\end{array}
 $$
 $$
=1 + \sigma(l) + \tau_2((l+1)/2) + \sigma(2n+l) - \sigma(2n-1) + \tau_2(n-1)
$$
Now,~\eqref{eq_degh} follows from $\tau_2((l+1)/2) = \tau_2(l+1)$ and
$\tau_2(2n-1) = 1+ \tau_2(n-1)$.

The last case of $\det H(2n+1, l+1)$ is done analogously. By~\eqref{eq_hnl},
we have
$$
\det H(2n+1, l+1) = \pm \det H(n+1,(l+1)/2) \cdot \det \tilde{H} (n+1, (l-1)/2).
$$
By inductional assumption, both factors on the right hand side are
monic, therefore so is $\det H(2n+1, l+1)$. The degree of this
polynomial is
$$
\begin{array}{rl}
\deg(\det H(2n+1, l+1)) = &\sigma((l+1)/2)+ \sigma(n+(l+1)/2) - \sigma(n-1)\\
 + &l+1 + \sigma((l-1)/2)+\sigma(n+(l-1)/2) - \sigma(n-1)\\
 = &\sigma(l+1) + \sigma(2n+l+1) - \sigma(2n-1).
\end{array}
$$

This finishes the induction.

We have shown that all the polynomials $\det H(n,l)$ are monic. We will now
show that they are in fact doubly monic. Given $n,l\in \ZZ_{\ge 0}$, choose
$d$ large enough such that $2^d\ge n+2l$. Note that for any $1\le m\le 2^d$
one has $a_n = u^d a_{2^d+1-n}^{-1}$, i.e. the sequence $a_{2^d}, a_{2^d-1},
\ldots, a_1$ coincides with the sequence of Laurent series coefficients
$a^*_1, a^*_2, \ldots, a^*_{2^d}$ for the function $u^d g_{1/u}(t)$. Hence
$$
|\det H_{g_u}(n,l)| = |u^{d(l+1)} \cdot \det H_{g_{1/u}}(2^d+1-n,l)|.
$$
Therefore, the smallest non-zero coefficient of $H_{g_u}(n,l)$ coincides with
the leading non-zero coefficient of $H_{g_{1/u}}(2^d+1-n,l)$ as a polynomial
of $u^{-1}$. Since the latter polynomial is monic, the polynomial
$H_{g_u}(n,l)$ is doubly monic. This finishes the proof of Theorem~\ref{th5}.

\endproof

\section{Finite fields $\FF$}

Consider the generalised Thue-Morse function $g_u(t)$ over a finite field
$\FF$. We should have $u\neq 0$ because otherwise $g_u=1$ is a rational
function. In the proof of Theorem~\ref{th5} we computed the Hankel
determinants $H(n,l)$ of $g_u$ for $l=2$. Obviously the same formulae remain
true for any fields. In particular, for even values of $n$ one has
$$
H(n,1) = u^{2\tau_2(n)} (u^{-\nu_2(n)} - 1).
$$
Choose $n$ so that $-\nu_2(n) = \ord(u)$. Then we have that $H(n,1)$ vanishes
and by Proposition~\ref{prop1} we get that $g_u\not\in\elc$. In other words,
for any finite field $\FF$ and for any $u\in\FF$ the Laurent series $g_u$
satisfies $t$-adic Littlewood conjecture. This finishes the proof of
Theorem~\ref{th4}.


\begin{thebibliography}{99}

\bibitem{ad_ne_lu_2019} Adiceam F., Nesharim E. and Lunnon F., ``On the
    $t$-adic Littlewood conjecture'', preprint.

\bibitem{al_pe_we_we_1998} Allouche J.-P., Peyri\`ere J., Wen Z. X. and Wen
    Z. Y., ``Hankel determinants of the Thue-Morse sequence'', Ann. Inst.
    Fourier, \textbf{48}(1) (1998), 1--27.

\bibitem{badziahin_2018} Badziahin D., ``Continued fractions of certain
    Mahler functions'', Acta Arith., \textbf{188}(1) (2019), 53--81.

\bibitem{bugeaud_2014} Bugeaud, Y. ``Around the Littlewood conjecture in Diophantine
approximation.'', Publications math'ematiques de Besançon (2014),
5-18.

\bibitem{bu_ha_we_ya_2016} Bugeaud Y., Han G. N., Wen Z. Y. and Yao J. Y.,
    ``Hankel determinants, Pad\'e approximations, and irrationality
    exponents'', IMRN \textbf{2016}(5) (2016), 1467--1496.

\bibitem{bug_mat_2008} Bugeaud Y., de Mathan B., ``On a mixed Littlewood
    conjecture in fieds of power series'', AIP Conf. Proc., \textbf{976}, 19
    (2008).

\bibitem{coons_2013} Coons M., ``On the rational approximation of the sum of
    the reciprocals of the Fermat numbers'', Ramanujan J., \textbf{30}(1)
    (2013), 39--65.

\bibitem{han_2015a} Han, G. N., ``Hankel determinant calculus for the
    Thue-Morse and related sequences'', J. Numb. Theor. \textbf{147} (2015),
    374 -- 395.

\bibitem{han_2015} Han G. N., ``Hankel continued fraction and its
    applications'', Adv, in Math. \textbf{303} (2016), 295--321.

\bibitem{mat_teu_2004} de Mathan B. and Teuli\'e O., ``Probl\`emes
    diophantiens simultan\'es'', Monatsh. Math., \textbf{143}(3) (2007),
    229--245.

\bibitem{poorten_1998} Van der Poorten A. J., ``Formal power series and
    their continued fraction expansion'', Algorithmic Number Theory,  Lecture notes in Computer Science, 1423,
Springer, Berlin (1998), 358 -- 371.


\bibitem{dav_schmidt_1969} Davenport H. and Schmidt W., ``Approximation to
    real numbers by algebraic integers'', Acta Arith., 15 (1969):
    393--416.
\end{thebibliography}
\end{document}